\documentclass[12pt,reqno]{article}

\usepackage[usenames]{color}

\usepackage{tikz}
\tikzset{
    state/.style={
           rectangle,
           rounded corners,
           draw=black, very thick,
           minimum height=1.25em,
           inner sep=2pt,
           text centered,
           },
}

\tikzset{
    state2/.style={
           rectangle, dotted,
           draw=black,
           minimum height=1.15em,
           inner sep=2pt,
           text centered,
           fill=white
           },
}

\usepackage[colorlinks=true,
linkcolor=webgreen,
filecolor=webbrown,
citecolor=webgreen]{hyperref}

\definecolor{webgreen}{rgb}{0,.5,0}
\definecolor{webbrown}{rgb}{.6,0,0}

\usepackage{fullpage}
\usepackage{float}

\usepackage{graphicx,amsmath,amssymb}
\usepackage{amsthm}
\usepackage{amsfonts}
\usepackage{latexsym}
\usepackage{epsf}

\newcommand{\seqnum}[1]{\href{http://oeis.org/#1}{\underline{#1}}}

\DeclareMathOperator{\pr}{\mathfrak{pr}}
\DeclareMathOperator{\popop}{\mathfrak{popop}}

\begin{document}


\theoremstyle{plain}
\newtheorem{theorem}{Theorem}
\newtheorem{corollary}[theorem]{Corollary}
\newtheorem{lemma}[theorem]{Lemma}
\newtheorem{proposition}[theorem]{Proposition}
\newtheorem{conjecture}[theorem]{Conjecture}
\newtheorem{claim}[theorem]{Claim}
\newtheorem*{theorem*}{Theorem}

\theoremstyle{definition}
\newtheorem{observation}[theorem]{Observation}
\newtheorem{definition}[theorem]{Definition}
\newtheorem{question}[theorem]{Question}
\newtheorem{open_question}[theorem]{Open Question}
\newtheorem{example}[theorem]{Example}

\theoremstyle{remark}
\newtheorem{remark}[theorem]{Remark}

\begin{center}
\vskip 1cm{\LARGE\bf Coprime and Prime Labelings of Graphs}
\vskip 1cm
Adam H. Berliner, Department of Mathematics, Statistics, and Computer Science\\
St.~Olaf College, Northfield, MN 55057, USA\\
\href{mailto:berliner@stolaf.edu}{\tt berliner@stolaf.edu} \\
\ \\
Nathaniel Dean, Department of Mathematics\\
Texas State University, San Marcos, TX  78666, USA\\
\href{mailto:nd17@txstate.edu}{\tt nd17@txstate.edu}\\
\ \\
Jonelle Hook, Department of Mathematics and Computer Science\\
Mount St.~Mary's University, Emmitsburg, MD 21727, USA\\
\href{mailto:jhook@msmary.edu}{\tt jhook@msmary.edu}\\
\ \\
Alison Marr, Department of Mathematics and Computer Science\\
Southwestern University, Georgetown, TX 78626, USA\\
\href{mailto:marra@southwestern.edu}{\tt marra@southwestern.edu}\\
\ \\
Aba Mbirika, Department of Mathematics\\
University of Wisconsin-Eau Claire, Eau Claire, WI  54702, USA\\
\href{mailto:mbirika@uwec.edu}{\tt mbirika@uwec.edu}\\
\ \\
Cayla D. McBee, Department of Mathematics and Computer Science\\
Providence College, Providence, RI 02918, USA\\
\href{mailto:cmcbee@providence.edu}{\tt cmcbee@providence.edu}
\end{center}

\vskip .2 in

\begin{abstract}   
A coprime labeling of a simple graph of order $n$ is a labeling in which adjacent vertices are given relatively prime labels, and a graph is prime if the labels used can be taken to be the first $n$ positive integers.  In this paper, we consider when ladder graphs are prime and when the corresponding labeling may be done in a cyclic manner around the vertices of the ladder.  Furthermore, we discuss coprime labelings for complete bipartite graphs. 
\end{abstract}

\section{Introduction}\label{sec:Intro}

Let $G=(V,E)$ be a simple graph with vertex set $V$ and edge set $E$, where $n = |V|$ is the number of vertices of $G$.  A \textit{coprime labeling} of $G$ is a labeling of the vertices of $G$ with distinct integers from the set $\{1, 2, \ldots, k\}$, for some $k \geq n$, in such a way that the labels of any two adjacent vertices are relatively prime.  We then define $\pr(G)$ to be the minimum value of $k$ for which $G$ has a coprime labeling.  The corresponding labeling of $G$ is called a \textit{minimal coprime labeling} of $G$.  

If $\pr(G)=n$, then a corresponding minimal coprime labeling of $G$ is called a \textit{prime labeling} of $G$ and we call $G$ \textit{prime}.  Though these definitions are more common, some of the literature uses the term \textit{coprime} to mean what we refer to as \textit{prime} in this paper (cf.~\cite{erdos-sarkozy96}).   In our setting it does not make sense to refer to a graph as {\em coprime}, since all graphs have a coprime labeling (for example, use the first $n$ prime integers as the labels).

Much work has been done on various types of labeling problems, including coprime and prime graphs (see \cite{Gallian} for a detailed survey).  For nearly 35 years, Entringer's conjecture that all trees are prime has remained unsolved.  Some progress towards the result has been made.  In 1994, Fu and Huang proved trees with 15 or fewer vertices are prime \cite{Fu-Huang94}.  Pikhurko improved the result for trees of up to 50 vertices \cite{Pikhurko07} in 2007.  In 2011, Haxell, Pikhurko, and Taraz  \cite{HPT11} prove Entringer's conjecture for trees of sufficiently large order.  Specifically, it is known that paths, stars, caterpillars, complete binary trees, and spiders are prime.

Many other classes of graphs have been studied as well, several of which are constructed from trees.  If we let $P_n$ denote the path on $n$ vertices, then the Cartesian product $P_n\times P_m$, where $m\leq n$, is called a {\it grid graph}.  Some results about prime labelings of grid graphs can be found in \cite{SPS06, Kanetkar}. If $m=2$, then the graph is called a {\it ladder}.  Several results are known about ladders. For example, if $n$ and $k$ are prime, then $P_{n} \times P_{2}$, $P_{n+1} \times P_{2}$, $P_{n+k} \times P_{2}$, $P_{3n} \times P_{2}$, and $P_{n+2} \times P_{2}$ are prime \cite{VSN, SPS06, SPS07}. Ladders $P_{n} \times P_{2}$, $P_{n+1} \times P_{2}$, and $P_{n+2} \times P_{2}$ have also been shown to be prime when $2n+1$ is prime \cite{VSN,SPS07,Varkey}. In \cite{Varkey}, it is conjectured that all ladders are prime. While we cite the papers \cite{VSN,Varkey} here, we note that both of them contain some errors and incomplete proofs. 

In Section~\ref{sec:ladders}, we further consider prime labelings for ladders.  Moreover, we consider instances when a prime labeling exists where the labels occur in numerical order around the vertices of the ladder.

In Section~\ref{sec:bipartite}, we consider complete bipartite graphs $K_{m,n}$.  In the case $m=n$, it is clear that $K_{n,n}$ is not prime when $n>2$.  Thus, we focus on minimal coprime labelings.  In the more general case of $m<n$, for each $m$ we give a sufficiently large lower bound value of $n$ for which $K_{m,n}$ is prime.  Specifically, we give all values of $n$ for which $K_{m,n}$ is prime for $m\leq 13$.

In Section~\ref{sec:future_work}, we conclude with some possible directions for future work.

\section{Ladders}\label{sec:ladders}
In this section, we give labelings of ladders that are mainly constructed in a cyclic manner. As mentioned above, several classes of ladders have been shown to be prime. We reproduce some of the results of \cite{VSN} and \cite{SPS07}, but our constructions are arguably more elegant and complete than those given. 

\begin{theorem}\label{thm:Deans_idea}
If $n+1$ is prime, then $P_n \times P_2$ has a prime labeling.  Moreover, this prime labeling can be realized with top row labels from left to right, $1, 2, \ldots, n$, and bottom row labels from left to right, $n+2, n+3, \ldots, 2n, n+1$.
\end{theorem}

\begin{proof}
Consider the graph $P_n \times P_2$ where $n+1$ is prime.  We claim that the following vertex labeling gives a prime labeling:
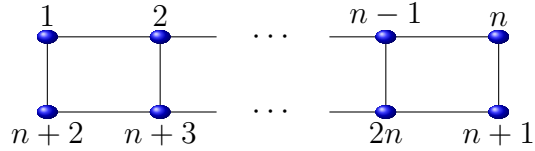
\begin{center}
\begin{tikzpicture}[xscale=1.5]
\draw (0,0) -- (1.5,0);
\draw (0,1) -- (1.5,1);
\draw (2.5,0) -- (4,0);
\draw (2.5,1) -- (4,1);
\foreach \x in {0,1,3,4}
		\draw (\x,0) -- (\x,1);
\node at (2,0) {$\cdots$};
\node at (2,1) {$\cdots$};
\foreach \x in {0,1,3,4}
		\shade[ball color=blue] (\x,0) circle (.5ex);
\foreach \x in {0,1,3,4}
		\shade[ball color=blue] (\x,1) circle (.5ex);
\node [above] at (0,1) {1};
\node [above] at (1,1) {2};
\node [above] at (3,1) {$n-1$};
\node [above] at (4,1) {$n$};
\node [below] at (0,0) {$n+2$};
\node [below] at (1,0) {$n+3$};
\node [below] at (3,0) {$2n$};
\node [below] at (4,0) {$n+1$};
\end{tikzpicture}
\end{center}
Since $\gcd(k,k+1) = 1$, it suffices to check only the vertex labels arising from the endpoints of the following $n$ particular edges:
\begin{itemize}
\item the horizontal edge connecting vertex labels $2n$ and $n+1$, and
\item the first $n-1$ vertical edges going from left to right.
\end{itemize}
Since $n+1$ is prime and $2n < 2(n+1)$, then $n+1$ cannot divide $2n$.  Hence, $\gcd(2n, n+1)=1$ as desired.  Observe that each of the  $n-1$ vertical edges under consideration have vertex labels $a$ and $(n+1)+a$ for $1 \leq a \leq n-1$.  It follows that $\gcd((n+1)+a,a) = \gcd(n+1,a) = 1$.  Thus, the graph $P_n \times P_2$ is prime whenever $n+1$ is prime.
\end{proof}

The remaining theorems involve consecutive cyclic prime labelings of ladders.  
Let $P_n \times P_2$ be the ladder with vertices $v_1, v_2, \dots, v_n$ and $u_1, u_2, \dots, u_n$, where $v_i$ is adjacent to $u_i$ for $1\leq i\leq n$, $v_i$ is adjacent to $v_{i+1}$ for $1\leq i\leq n-1$, and $u_i$ is adjacent to $u_{i+1}$ for $1\leq i\leq n-1$. When drawing a ladder graph, we may assume without loss of generality that $v_1$ denotes the top left vertex of the graph.  

\begin{definition}\label{def:ccpl}
A \textit{consecutive cyclic prime labeling} of a ladder $P_n \times P_2$ is a prime labeling in which the labels on the vertices wrap around the ladder in a consecutive way. In particular, if the label 1 is placed on vertex $v_i$, then 2 will be placed on $v_{i+1}$, $n-i+1$ will be placed on $v_n$, $n-i+2$ on $u_n$, $2n-i+1$ on $u_1$, $2n-i+2$ on $v_1$, and $2n$ will be placed on vertex $v_{i-1}$. A similar definition holds if 1 is placed on $u_i$.
\end{definition}

The reverse direction of the following theorem is stated and proved in \cite{VSN, Varkey}, but we include our own proof here for completeness. 

\begin{theorem}\label{iff_theorem}
$P_n \times P_2$ has a consecutive cyclic prime labeling with the value $1$ assigned to vertex $v_1$ if and only if $2n + 1$ is prime.
\end{theorem}

\begin{proof}
We prove the forward implication by contradiction. Let $p>1$ be a divisor of $2n+1$. The following consecutive cyclic labeling of $P_n \times P_2$ with the value of 1 assigned to the top left vertex of the graph is not a prime labeling. The pair of vertices labeled $p$ and $2n-(p-1)=2n+1 -p$ are not relatively prime since $p\ |\ 2n+1$.

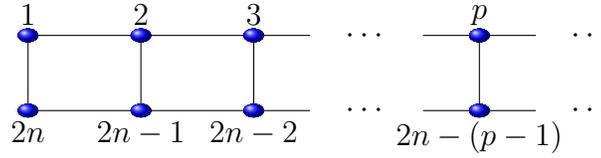
\begin{center}
\begin{tikzpicture}[xscale=1.5]
\draw (0,0) -- (2.5,0);
\draw (0,1) -- (2.5,1);
\draw (3.5,0) -- (4,0);
\draw (3.5,1) -- (4,1);
\draw (4,0) -- (4.5,0);
\draw (4,1) -- (4.5,1);
\foreach \x in {0,1,2,4}
		\draw (\x,0) -- (\x,1);
\node at (3,0) {$\cdots$};
\node at (3,1) {$\cdots$};
\foreach \x in {0,1,2,4}
		\shade[ball color=blue] (\x,0) circle (.5ex);
\foreach \x in {0,1,2,4}
		\shade[ball color=blue] (\x,1) circle (.5ex);
\node at (5,0) {$\cdots$};
\node at (5,1) {$\cdots$};
\node [above] at (0,1) {1};
\node [above] at (1,1) {2};
\node [above] at (2,1) {3};
\node [above] at (4,1) {$p$};
\node [below] at (0,0) {$2n$};
\node [below] at (1,0) {$2n-1$};
\node [below] at (2,0) {$2n-2$};
\node [below] at (4,0) {$2n-(p-1)$};
\end{tikzpicture}
\end{center}

Conversely, consider the graph $P_n \times P_2$ where $2n+1$ is prime.  We claim that the following vertex labeling gives a consecutive cyclic prime labeling:
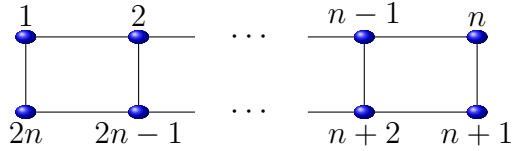
\begin{center}
\begin{tikzpicture}[xscale=1.5]
\draw (0,0) -- (1.5,0);
\draw (0,1) -- (1.5,1);
\draw (2.5,0) -- (4,0);
\draw (2.5,1) -- (4,1);
\foreach \x in {0,1,3,4}
		\draw (\x,0) -- (\x,1);
\node at (2,0) {$\cdots$};
\node at (2,1) {$\cdots$};
\foreach \x in {0,1,3,4}
		\shade[ball color=blue] (\x,0) circle (.5ex);
\foreach \x in {0,1,3,4}
		\shade[ball color=blue] (\x,1) circle (.5ex);
\node [above] at (0,1) {1};
\node [above] at (1,1) {2};
\node [above] at (3,1) {$n-1$};
\node [above] at (4,1) {$n$};
\node [below] at (0,0) {$2n$};
\node [below] at (1,0) {$2n-1$};
\node [below] at (3,0) {$n+2$};
\node [below] at (4,0) {$n+1$};
\end{tikzpicture}
\end{center}
Since $\gcd(k,k+1) = 1$, it suffices to check only the vertex labels arising from the endpoints of the first $n-1$ vertical edges going from left to right.  Observe that each of the  $n-1$ vertical edges under consideration have vertex labels $a$ and $(2n+1)-a$ for $1 \leq a \leq n-1$.  We conclude that
$$\gcd(a,(2n+1)-a) = \gcd(a,2n+1) = 1.$$
Thus, the graph $P_n \times P_2$ has a consecutive cyclic prime labeling whenever $2n+1$ is prime. 
\end{proof}

When drawing a ladder, $n$ columns are formed consisting of a vertex from the first path, a vertex from the second path, and the edge between them. When we place labels on the vertices, we create $n$ column sums which are just the sum of the label on vertex $u_i$ and $v_i$ for $1\leq i\leq n$.  When constructing a consecutive cyclic labeling, without loss of generality we place a value of 1 somewhere in the top row of the ladder and increase each next vertex label by one in a clockwise direction. Thus, there will be a value $k$ directly below the 1 depending on where the 1 is placed in the top row. This creates column sums of  $k+1$ for columns to the right of 1 (and including the column with a 1) and column sums of $2n+k+1$ for columns to the left of the 1. 

\begin{theorem}\label{pcolsumsmod2n_theorem}
For every consecutive cyclic labeling of $P_n\times P_2$, the column sums are congruent to $k+1$ modulo $2n$, where $k$ is the label on the vertex directly below the vertex with label 1. If $k+1$ is not prime, then the consecutive cyclic labeling is not a prime labeling.
\end{theorem}

\begin{proof}
Consider a ladder graph with the following consecutive cyclic labeling.

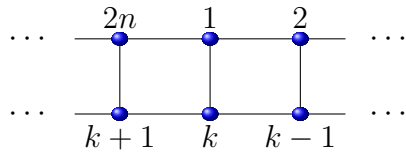
\begin{center}
\begin{tikzpicture}[xscale=1.2]
\draw (0.5,0) -- (3.5,0);
\draw (0.5,1) -- (3.5,1);
\node at (0,0) {$\cdots$};
\node at (0,1) {$\cdots$};
\node at (4,0) {$\cdots$};
\node at (4,1) {$\cdots$};
\foreach \x in {1,2,3}
		\draw (\x,0) -- (\x,1);
\foreach \x in {1,2,3}
		\shade[ball color=blue] (\x,0) circle (.5ex);
\foreach \x in {1,2,3}
		\shade[ball color=blue] (\x,1) circle (.5ex);
\node [above] at (1,1) {$2n$};
\node [above] at (2,1) {1};
\node [above] at (3,1) {2};
\node [below] at (1,0) {$k+1$};
\node [below] at (2,0) {$k$};
\node [below] at (3,0) {$k-1$};
\end{tikzpicture}
\end{center}

First, recall that each column sum in a consecutive cyclic labeling will be either $k+1$ or $2n+k+1$.
If $k+1$ is prime, then all the column sums are congruent to a prime modulo $2n$. So $k+1$ must be composite (which is an odd composite since $k$ must be even). Then $q\ |\ k+1$ for some prime $q$. Note that $k+1=q\cdot s > 2q$ which implies that $q < \dfrac{k+1}{2}$. In the above consecutive cyclic labeling, the labels $q$ and $k+1-q$ are (vertically) adjacent. However, gcd$(q,\ k+1-q)=q$ and so the consecutive cyclic labeling is not a prime labeling.
\end{proof}

The converse of Theorem \ref{pcolsumsmod2n_theorem} does not hold. If $k+1$ is prime, then it does not guarantee there exists a consecutive cyclic prime labeling, as the following example illustrates. 
\begin{example}
Below is a consecutive cyclic labeling of $P_5\times P_2$ where $k+1$ equals $5$, but the labeling is not prime.

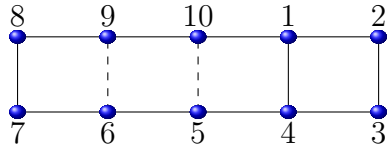
\begin{center}
\begin{tikzpicture}[xscale=1.2]
\draw (0,0) -- (4,0);
\draw (0,1) -- (4,1);
\foreach \x in {0,3,4}
		\draw (\x,0) -- (\x,1);
\draw[dashed] (1,0)--(1,1);
\draw[dashed] (2,0)--(2,1);
\foreach \x in {0,1,2,3,4}
		\shade[ball color=blue] (\x,0) circle (.5ex);
\foreach \x in {0,1,2,3,4}
		\shade[ball color=blue] (\x,1) circle (.5ex);
\node [above] at (0,1) {8};
\node [above] at (1,1) {9};
\node [above] at (2,1) {10};
\node [above] at (3,1) {1};
\node [above] at (4,1) {2};
\node [below] at (0,0) {7};
\node [below] at (1,0) {6};
\node [below] at (2,0) {5};
\node [below] at (3,0) {4};
\node [below] at (4,0) {3};
\end{tikzpicture}
\end{center}
\end{example}

\begin{lemma}\label{primecolsums} Consider a consecutive cyclic labeling of $P_n\times P_2$ and let $k$ be the label on the vertex directly below the vertex with label 1. If the column sums $k+1$ and $2n+k+1$ are prime, then the labeling is a consecutive cyclic prime labeling of $P_n \times P_2$.
\end{lemma}

\begin{proof}
Consider the graph $P_n \times P_2$ with a consecutive cyclic labeling and prime column sums $k+1$ and $2n+k+1$.  To show this labeling is prime we must verify that the vertex labels arising from the endpoints of the vertical edges are relatively prime.  Letting $k+1=p$ results in vertical pairs of labels $(a, p-a)$ for $1 \leq a \leq \frac{1}{2}(p-1)$ and $(2n-b, p+b)$ for $0 \leq b \leq n-\frac{1}{2}(p-1)$.  

First we claim that gcd($a$, $p-a$)=1.  Suppose gcd($a$, $p-a$)=$d$ for $d$ a positive integer.  This implies $d\ |\ (p-a+a)$ and so $d\ |\ p$.  Therefore, $d=1$ or $p$, but, $d \neq p$ since $d \leq a < p$. Thus, gcd($a$, $p-a$)=1.        

Next we claim that gcd($2n-b$, $p+b$)=1.  Suppose gcd($2n-b$, $p+b$)=$d$ for $d$ a positive integer.  Then $d\ |\ ((2n-b) + (p+b))$ which is equivalent to $d\ |\ (2n+p)$.  Since we assumed $2n+p$ is prime, $d=1$ or $2n+p$.  However, $d \neq 2n+p$ since $d \leq 2n-b < 2n+p$.  Therefore, gcd($2n-b$, $p+b$)=1. 

Given the cyclic labeling of the graph, all horizontal edges connect labels that are relatively prime.  Thus we can conclude that if the column sums $k+1$ and $2n+k+1$ are prime, then the labeling is a consecutive cyclic prime labeling.
\end{proof}

The converse of Lemma~\ref{primecolsums} is not true, as is shown in the following example. 

\begin{example}
The labeling below is a consecutive cyclic prime labeling of $P_4\times P_2$ with column sums 7 and 15.

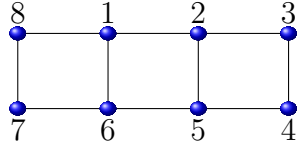
\begin{center}
\begin{tikzpicture}[xscale=1.2]
\draw (0,0) -- (3,0);
\draw (0,1) -- (3,1);
\foreach \x in {0,1,2,3}
		\draw (\x,0) -- (\x,1);
\foreach \x in {0,1,2,3}
		\shade[ball color=blue] (\x,0) circle (.5ex);
\foreach \x in {0,1,2,3}
		\shade[ball color=blue] (\x,1) circle (.5ex);
\node [above] at (0,1) {8};
\node [above] at (1,1) {1};
\node [above] at (2,1) {2};
\node [above] at (3,1) {3};
\node [below] at (0,0) {7};
\node [below] at (1,0) {6};
\node [below] at (2,0) {5};
\node [below] at (3,0) {4};
\end{tikzpicture}
\end{center}
\end{example}

\begin{theorem}\label{big_theorem}
If $2n+p$ is prime where $p$ is a prime less than $2n+1$, then $P_n \times P_2$ has a consecutive cyclic prime labeling.  Moreover, this labeling can be realized by assigning $1$ to the vertex in the location $\frac{1}{2}(p-1)-1$ places from the top right vertex.
\end{theorem}

\begin{proof}
Let $2n+p$ be a prime where $p$ is prime less than $2n+1$.  We claim that the following gives a consecutive cyclic labeling of $P_n \times P_2$:
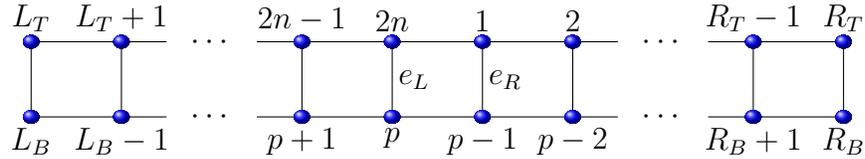
\begin{center}
\begin{tikzpicture}[xscale=1.2]
\draw (0,0) -- (1.5,0);
\draw (0,1) -- (1.5,1);
\draw (2.5,0) -- (6.5,0);
\draw (2.5,1) -- (6.5,1);
\draw (7.5,0) -- (9,0);
\draw (7.5,1) -- (9,1);
\foreach \x in {0,1,3,4,5,6,8,9}
		\draw (\x,0) -- (\x,1);
\node at (2,0) {$\cdots$};
\node at (2,1) {$\cdots$};
\node at (7,0) {$\cdots$};
\node at (7,1) {$\cdots$};
\node [right] at (3.95,.5) {$e_L$};
\node [right] at (4.95,.5) {$e_R$};
\foreach \x in {0,1,3,4,5,6,8,9}
		\shade[ball color=blue] (\x,0) circle (.5ex);
\foreach \x in {0,1,3,4,5,6,8,9}
		\shade[ball color=blue] (\x,1) circle (.5ex);
\node [above] at (0,1) {$L_T$};
\node [above] at (1,1) {$L_T+1$};
\node [above] at (3,1) {$2n-1$};
\node [above] at (4,1) {$2n$};
\node [above] at (5,1) {$1$};
\node [above] at (6,1) {$2$};
\node [above] at (8,1) {$R_T-1$};
\node [above] at (9,1) {$R_T$};
\node [below] at (0,0) {$L_B$};
\node [below] at (1,0) {$L_B-1$};
\node [below] at (3,0) {$p+1$};
\node [below] at (4,0) {$p$};
\node [below] at (5,0) {$p-1$};
\node [below] at (6,0) {$p-2$};
\node [below] at (8,0) {$R_B+1$};
\node [below] at (9,0) {$R_B$};
\end{tikzpicture}
\end{center}
where the four corner values are
\begin{align*}
L_T &= \frac{1}{2}(p+1)+n & \hspace{.5in} R_T &= \frac{1}{2}(p-1)\\
L_B &= \frac{1}{2}(p-1)+n & \hspace{.5in} R_B &= \frac{1}{2}(p-1)+1.
\end{align*}
Since $\gcd(k,k+1)=1$, it suffices to check only the vertex labels arising from the endpoints of the first $n-1$ vertical edges going from left to right.  For the vertical edges right of (and including) $e_R$, the column sum is $p$.  For the vertical edges left of (and including) $e_L$, the column sum is $2n+p$.  Since both of these column sums are prime, Lemma~\ref{primecolsums} implies the ladder has a consecutive cyclic prime labeling. 
\end{proof}

It is of interest to know if there exists a prime number of the form $2n+p$ where $n$ is an integer and $p = 1$ or $p$ is a prime less than $2n+1$.  If such a prime exists, then by Theorem~\ref{iff_theorem} and Theorem~\ref{big_theorem} we may conclude that $\pr(P_n \times P_2) = 2n$ for all $n$ and also that every ladder has a consecutive cyclic prime labeling that we may easily construct.  

Unfortunately, determining whether such a prime exists is a difficult problem related to Polignac's conjecture.  Polignac's conjecture, first stated by Alphonse de Polignac in 1849~\cite{Polignac}, states that for any positive even integer $n$, there are infinitely many prime gaps of size $n$.  If $n=2$, the conjecture is equivalent to the twin prime conjecture.  In looking for cyclic prime labelings of ladder graphs, we are interested in finding pairs of primes that differ by the even number $2n$ where the smaller prime is less than $2n+1$.  Although Polignac's conjecture guarantees the existence of pairs of primes whose difference is $2n$, the result remains as yet to be proved.  Also, Polignac's conjecture does not address our additional constraint that the smaller prime be less than $2n+1$. In summary, we have the following observation:

\begin{observation} If every even integer $2n$ can be written in the form $2n=q-p$ where $q$ is a prime and $p$ is either 1 or a prime less than $2n+1$, then all ladder graphs are prime. \end{observation}

The labelings in the following example illustrate our results thus far.

\begin{example}
Consider $P_{10} \times P_2$.  Theorem~\ref{iff_theorem} does not apply because $2n+1 = 21$ is not prime.  Consequently, assigning 1 to the top left vertex does not yield a consecutive cyclic prime labeling.   

The only primes $p$ for which $2n+p$ is prime and $p<2n+1$ are $p = 3,11,17$. Thus, Theorem~\ref{big_theorem} holds and we have the following consecutive cyclic prime labelings:
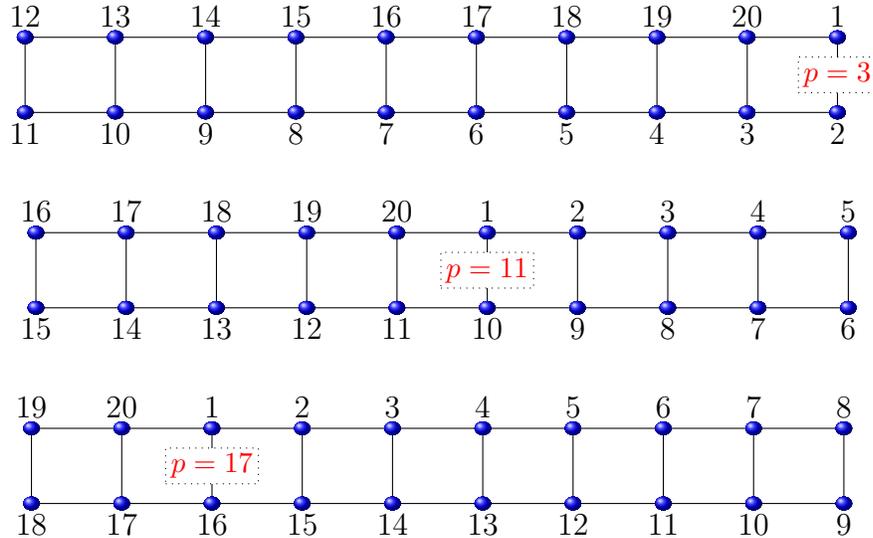
\begin{center}
\begin{tikzpicture}[xscale=1.2]
\draw (0,0) -- (9,0); \draw (0,1) -- (9,1); \foreach \x in {0,...,9} \draw (\x,0) -- (\x,1);
\node [state2] at (9,.5) {\textcolor{red}{\text{\small$p=3$}}};
\node [above] at (0,1) {12}; \node [below] at (0,0) {11};
\node [above] at (1,1) {13}; \node [below] at (1,0) {10};
\node [above] at (2,1) {14}; \node [below] at (2,0) {9};
\node [above] at (3,1) {15}; \node [below] at (3,0) {8};
\node [above] at (4,1) {16}; \node [below] at (4,0) {7};
\node [above] at (5,1) {17}; \node [below] at (5,0) {6};
\node [above] at (6,1) {18}; \node [below] at (6,0) {5};
\node [above] at (7,1) {19}; \node [below] at (7,0) {4};
\node [above] at (8,1) {20}; \node [below] at (8,0) {3};
\node [above] at (9,1) {1}; \node [below] at (9,0) {2};
\foreach \x in {0,...,9} \shade[ball color=blue] (\x,0) circle (.5ex); \foreach \x in {0,...,9} \shade[ball color=blue] (\x,1) circle (.5ex);
\end{tikzpicture}
\end{center}

\begin{center}
\begin{tikzpicture}[xscale=1.2]
\draw (0,0) -- (9,0); \draw (0,1) -- (9,1); \foreach \x in {0,...,9} \draw (\x,0) -- (\x,1);
\node [state2] at (5,.5) {\textcolor{red}{\text{\small$p=11$}}};
\node [above] at (0,1) {16}; \node [below] at (0,0) {15};
\node [above] at (1,1) {17}; \node [below] at (1,0) {14};
\node [above] at (2,1) {18}; \node [below] at (2,0) {13};
\node [above] at (3,1) {19}; \node [below] at (3,0) {12};
\node [above] at (4,1) {20}; \node [below] at (4,0) {11};
\node [above] at (5,1) {1}; \node [below] at (5,0) {10};
\node [above] at (6,1) {2}; \node [below] at (6,0) {9};
\node [above] at (7,1) {3}; \node [below] at (7,0) {8};
\node [above] at (8,1) {4}; \node [below] at (8,0) {7};
\node [above] at (9,1) {5}; \node [below] at (9,0) {6};
\foreach \x in {0,...,9} \shade[ball color=blue] (\x,0) circle (.5ex); \foreach \x in {0,...,9} \shade[ball color=blue] (\x,1) circle (.5ex);
\end{tikzpicture}
\end{center}

\begin{center}
\begin{tikzpicture}[xscale=1.2]
\draw (0,0) -- (9,0); \draw (0,1) -- (9,1); \foreach \x in {0,...,9} \draw (\x,0) -- (\x,1);
\node [state2] at (2,.5) {\textcolor{red}{\text{\small$p=17$}}};
\node [above] at (0,1) {19}; \node [below] at (0,0) {18};
\node [above] at (1,1) {20}; \node [below] at (1,0) {17};
\node [above] at (2,1) {1}; \node [below] at (2,0) {16};
\node [above] at (3,1) {2}; \node [below] at (3,0) {15};
\node [above] at (4,1) {3}; \node [below] at (4,0) {14};
\node [above] at (5,1) {4}; \node [below] at (5,0) {13};
\node [above] at (6,1) {5}; \node [below] at (6,0) {12};
\node [above] at (7,1) {6}; \node [below] at (7,0) {11};
\node [above] at (8,1) {7}; \node [below] at (8,0) {10};
\node [above] at (9,1) {8}; \node [below] at (9,0) {9};
\foreach \x in {0,...,9} \shade[ball color=blue] (\x,0) circle (.5ex); \foreach \x in {0,...,9} \shade[ball color=blue] (\x,1) circle (.5ex);
\end{tikzpicture}.
\end{center}

In each graph, we highlight the value of $p$ where the location of the label 1 is determined by Theorem~\ref{big_theorem}.  For $p=19$, we observe that $2n+p = 39$ is not prime, yet the following labeling shows that assigning 1 to the prescribed vertex gives a successful labeling.

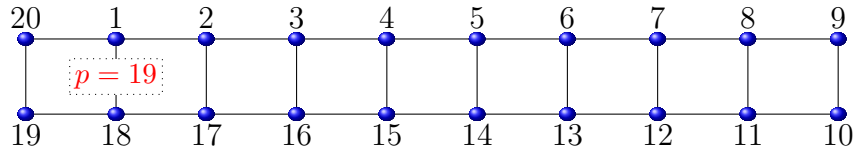
\begin{center}
\begin{tikzpicture}[xscale=1.2]
\draw (0,0) -- (9,0); \draw (0,1) -- (9,1); \foreach \x in {0,...,9} \draw (\x,0) -- (\x,1);
\node [state2] at (1,.5) {\textcolor{red}{\text{\small$p=19$}}};
\node [above] at (0,1) {20}; \node [below] at (0,0) {19};
\node [above] at (1,1) {1}; \node [below] at (1,0) {18};
\node [above] at (2,1) {2}; \node [below] at (2,0) {17};
\node [above] at (3,1) {3}; \node [below] at (3,0) {16};
\node [above] at (4,1) {4}; \node [below] at (4,0) {15};
\node [above] at (5,1) {5}; \node [below] at (5,0) {14};
\node [above] at (6,1) {6}; \node [below] at (6,0) {13};
\node [above] at (7,1) {7}; \node [below] at (7,0) {12};
\node [above] at (8,1) {8}; \node [below] at (8,0) {11};
\node [above] at (9,1) {9}; \node [below] at (9,0) {10};
\foreach \x in {0,...,9} \shade[ball color=blue] (\x,0) circle (.5ex); \foreach \x in {0,...,9} \shade[ball color=blue] (\x,1) circle (.5ex);
\end{tikzpicture}
\end{center}
An exhaustive check shows that assigning 1 to any other vertex in the top row fails to yield a consecutive cyclic prime labeling.
\end{example}

The previous example proves that the converse of Theorem~\ref{big_theorem} does not hold.  That is, there are primes $p$ for which $2n+p$ is not prime, yet the prescribed labeling is successful.

\section{Complete bipartite graphs}\label{sec:bipartite}

In this section, we look at prime labelings and coprime labelings of complete bipartite graphs.  We first examine minimal coprime labelings of $K_{n,n}$  for $n>2$.  Then, we consider complete bipartite graphs $K_{m,n}$ with $m<n$, which have prime labelings for sufficiently large $n$ (depending on the value of $m$).

\subsection{Minimal coprime labelings of \texorpdfstring{$K_{n,n}$}{Kn,n}}

It is straightforward to see that the complete bipartite graph $K_{n,n}$ has no prime labeling for $n>2$.  Hence, the best that one can do is to find the minimal value $\pr(K_{n,n}) > 2n$ such that $2n$ distinct labels chosen from the set $\{ 1, 2, \ldots, \pr(K_{n,n}) \}$ allows a coprime labeling of $K_{n,n}$.  We use the term \textit{minimal coprime labeling} to denote the latter assignment of labels on a graph $G$.

\begin{example}
The minimal coprime labelings of $K_{3,3}$ and $K_{4,4}$ below show that $\pr(K_{3,3})=7$ and $\pr(K_{4,4})=9$.

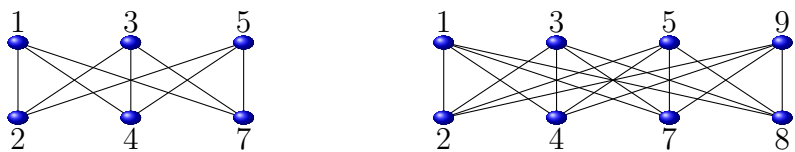
\begin{center}
\begin{tikzpicture}[xscale=1.5]
\foreach \x in {0,1,2}
\foreach \w in {0,1,2}
		\draw (\x,1) -- (\w,0);
\foreach \x in {0,1,2}
		\shade[ball color=blue] (\x,0) circle (.5ex);
\foreach \x in {0,1,2}
		\shade[ball color=blue] (\x,1) circle (.5ex);
\node [above] at (0,1) {1};
\node [above] at (1,1) {3};
\node [above] at (2,1) {5};
\node [below] at (0,0) {2};
\node [below] at (1,0) {4};
\node [below] at (2,0) {7};
\end{tikzpicture}
\hspace{.75in}
\begin{tikzpicture}[xscale=1.5]
\foreach \x in {0,1,2,3}
\foreach \w in {0,1,2,3}
		\draw (\x,1) -- (\w,0);
\foreach \x in {0,1,2,3}
		\shade[ball color=blue] (\x,0) circle (.5ex);
\foreach \x in {0,1,2,3}
		\shade[ball color=blue] (\x,1) circle (.5ex);
\node [above] at (0,1) {1};
\node [above] at (1,1) {3};
\node [above] at (2,1) {5};
\node [above] at (3,1) {9};
\node [below] at (0,0) {2};
\node [below] at (1,0) {4};
\node [below] at (2,0) {7};
\node [below] at (3,0) {8};
\end{tikzpicture}
\end{center}
\end{example}

Using an exhaustive computer check, we give the following  values for $\pr(K_{n,n})$.
\begin{center}
	\begin{tabular}{|c||ccccccccccccc|}
	\hline
	$n$ & 1 & 2 & 3 & 4 & 5 & 6 & 7 & 8 & 9 & 10 & 11 & 12 & 13 \\ \hline
	$\pr(K_{n,n})$ & 2 & 4 & 7 & 9 & 11 & 15 & 17 & 21 & 23 & 27 & 29 & 32 & 37 \\
	\hline
	\end{tabular}
\end{center}
This sequence~\seqnum{A213273} appears in the \textit{On-Line Encyclopedia of Integer Sequences} (OEIS)~\cite{Sloane}.  By further computer check, Alois Heinz extended the sequence to $n=23$, thus adding the following values for $\pr(K_{n,n})$:
\begin{center}
	\begin{tabular}{|c||cccccccccc|}
	\hline
	$n$ & 14 & 15 & 16 & 17 & 18 & 19 & 20 & 21 & 22 & 23 \\ \hline
	$\pr(K_{n,n})$ & 40 & 43 & 46 & 49 & 53 & 57 & 61 & 63 & 67 & 71\\
	\hline
	\end{tabular}
\end{center}

\begin{observation}
By analyzing the minimal coprime labelings constructed for $K_{n,n}$, the following facts are readily verified for $n\leq 13$:
\begin{itemize}
\item There exists a minimal coprime labeling with the labels 1 and 2 in separate partite sets.
\item All primes up to $\pr(K_{n,n})$ are used in a minimal coprime labeling.
\item There are never more than $\frac{n}{2}$ primes in a minimal coprime labeling of $K_{\frac{n}{2},\frac{n}{2}}$.
\item A small set, $\mathfrak{P}$, of \textit{carefully chosen} primes determine the minimal coprime labeling of $K_{n,n}$ by labeling one of the two partite sets of vertices with numbers from a set of products of powers of these primes (see Definition~\ref{popop}).
\end{itemize}
\end{observation}

\begin{definition}\label{popop}
Let $\mathfrak{P} = \{p_1, \ldots, p_j\}$ be a set of primes.  From $\mathfrak{P}$, we build the set of the first $n$ integers (larger than 1) of the form $p_1^{k_1} p_2^{k_2} \cdots p_j^{k_j}$ such that $k_i\geq 0$ for all $i$.  Since this set is constructed by taking \textit{products of powers of primes}, we denote this set of $n$ elements as $\popop(\mathfrak{P},n)$.  
\end{definition}

For example, for $\mathfrak{P}=\{2,3\}$,
$$\popop(\mathfrak{P},9) = \{ 2, 3, 4, 6, 8, 9, 12, 16, 18 \}.$$

\begin{conjecture}\label{Knn_conjecture}
For $K_{n,n}$, there exists a set of prime numbers $\{p_1, \ldots, p_j\}$ such that this set determines the values in the label sets of the two partite sets of vertices, giving $K_{n,n}$ a minimal coprime labeling.
\end{conjecture}

We believe the conjecture is true if we do the following.  Consider a \textit{carefully}\footnote{Currently, how to \textit{carefully} choose a prime set is not clear.  For example, $\popop(\{2,3\},n)$ allows a coprime labeling of $K_{n,n}$ for $n=3,8,13$, but fails for every other $n \leq 13$.  Whereas, $\popop(\{2,7,11,13\},n)$ allows a coprime labeling of $K_{n,n}$ for $n=3,4,10,11,12$, but fails for every other $n \leq13$.} chosen set of small primes, $\mathfrak{P}=\{p_1, \ldots, p_j\}$.   Let $A=\popop(\mathfrak{P},n)$ and $B$  contain the $n$ smallest positive integers which are relatively prime to all the elements of $A$.  We claim that the sets $A$ and $B$, when used to label the partite sets, yield a coprime labeling of $K_{n,n}$.  Moreover, we claim $\pr(K_{n,n}) = \max\{x \; | \; x \in A \cup B \}$, and hence this is a minimal coprime labeling.

\begin{example}
Via exhaustive search in \textsl{Mathematica}, we verified that there is a unique minimal coprime labeling of $K_{12,12}$ and  $\pr(K_{12,12}) = 32$.  If Conjecture~\ref{Knn_conjecture} is correct, then this calculation could have been done by the following method.  If we let $\mathfrak{P} = \{ 2, 7, 11, 13 \}$, then $\popop(\mathfrak{P},12)$ is the following:
$$A = \{ 2, 4, 7, 8, 11, 13, 14, 16, 22, 26, 28, 32 \}.$$
Hence, the set of the $n$ smallest positive integers all relatively prime to the elements of $A$ is:
$$B = \{ 1, 3, 5, 9, 15, 17, 19, 23, 25, 27, 29, 31 \}.$$
These two sets are exactly the sets in which the exhaustive computer check unveiled as the unique minimal coprime labeling of $K_{12,12}$.  Observe that the largest element in $A \cup B$ is 32, and indeed $\pr(K_{12,12}) = 32$.
\end{example}

Again using \textsl{Mathematica}, we observe that there exists a unique minimal coprime labeling of $K_{n,n}$ for $n = 1,2,5,9,11,12$.  On the other hand, for $n = 3,4,6,7,8,10,$ there are a variety of different minimal coprime labelings.  For example, $K_{8,8}$ has 5 different coprime labelings while $K_{10,10}$ has 9 different coprime labelings. Thus it is natural to ask if there is a way to determine the values of $n$ for which $K_{n,n}$ has a unique coprime labeling.  By further computation, Alois Heinz found the number of minimal coprime labelings of $K_{n,n}$ for $n\leq 23$, published as sequence \seqnum{A213806} in the OEIS~\cite{Sloane}.

\subsection{Prime labelings of \texorpdfstring{$K_{m,n}$}{Km,n}}

Although there exist no prime labelings for $K_{n,n}$ when $n>2$, there are prime labelings for $K_{m,n}$ when $m$ is fixed and $n$ is sufficiently large, depending on the value of $m$.  In 1990, Fu and Huang \cite{Fu-Huang94} proposed a necessary and sufficient condition for the graph $K_{m,n}$ to be prime.  Letting $P(t,v)$ be the set of all primes $x$ such that $t < x \leq v$, they prove the following proposition.

\begin{proposition} Let $m,n$ be positive integers where $m<n$.  Then
$K_{m,n}$ is prime if and only if  $m \leq \left|P\left(\frac{m+n}{2},m+n\right)\right|+1$.
\end{proposition}

We provide alternate proofs for specific cases of small values of $m$ and then in full generality. First, we must introduce some helpful notation and definitions.  Combining the notation of \cite{Fu-Huang94} with ours, we denote the set of labels for each partite set of vertices in a prime labeling of $K_{m,n}$ as $A_{m,n}$ and $B_{m,n}$.  For $K_{3,4}$, for example,

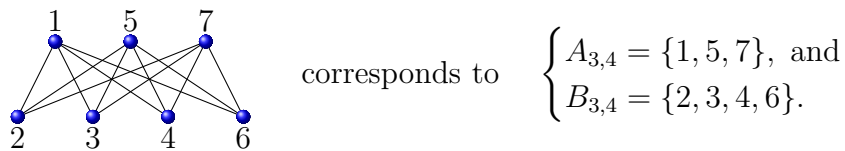
\begin{center}
\begin{tikzpicture}[baseline=2.5ex,xscale=1]
\foreach \x in {0.5,1.5,2.5}
\foreach \w in {0,1,2,3}
		\draw (\x,1) -- (\w,0);
\foreach \x in {0,1,2,3}
		\shade[ball color=blue] (\x,0) circle (.5ex);
\foreach \x in {0.5,1.5,2.5}
		\shade[ball color=blue] (\x,1) circle (.5ex);
\node [above] at (0.5,1) {1};
\node [above] at (1.5,1) {5};
\node [above] at (2.5,1) {7};
\node [below] at (0,0) {2};
\node [below] at (1,0) {3};
\node [below] at (2,0) {4};
\node [below] at (3,0) {6};
\end{tikzpicture}
\hspace{.1in} corresponds to \hspace{.1in}
$\begin{cases}
A_{3,4} = \{1,5,7\}, \mbox{ and}\\
B_{3,4} = \{2,3,4,6\}.
\end{cases}$
\end{center}

If we let $\pi(x)$ denote the number of primes less than or equal to $x$, then the {\it $n^{\text{th}}$ Ramanujan prime} is the least integer $R_n$ for which $\pi(x) - \pi(\frac{x}{2}) \geq n$ holds for all $x \geq R_n$.  The first few Ramanujan primes as given in sequence \seqnum{A104272} in the OEIS~\cite{Sloane} are as follows:

\begin{center}
	\begin{tabular}{|c||ccccc|}
	\hline
	$n$ & 1 & 2 & 3 & 4 & 5 \\ \hline
	$R_n$ & 2 & 11 & 17 & 29 & 41 \\
	\hline
	\end{tabular}
\end{center}

\begin{remark}
In his attempt to give a new proof of Bertrand's postulate in 1919, Ramanujan published a proof in which he not only proves the famous postulate, but also generalizes it to infinitely many cases. This answered the question, ``From which $x$ onward will there be at least $n$ primes lying between $\frac{x}{2}$ and $x$?''~\cite{ramanujan1919}.
\end{remark}

\begin{theorem}\label{main_bipartite_result}
$K_{m,n}$ is prime if $n \geq R_{m-1}-m$.
\end{theorem}

\begin{proof}
There are at least $m-1$ primes in the interval $(\frac{m+n}{2}, m+n]$ for $n\geq R_{m-1}-m$. If we denote the first $m-1$ of these primes by $p_1, p_2, \ldots , p_{m-1}$, then the sets
\begin{align*}
A_{m,n} &= \{1, p_1, p_2, \ldots, p_{m-1}\}\\
B_{m,n} &= \{1, \ldots, m+n\} \backslash A_{m,n}
\end{align*}
give a prime labeling of $K_{m,n}$ for $n \geq R_{m-1}-m$.
\end{proof}

\begin{remark}
Note that Theorem \ref{main_bipartite_result} provides a sufficiently large value of $n$ for which $K_{m,n}$ is always prime. It is interesting to note, however, that for certain smaller values of $n$, the graph $K_{m,n}$ still has a prime labeling.  These cases are summarized below for $3 \leq m \leq 13$.
\end{remark}

\begin{center}
	\begin{tabular}{|c||l|l|}
	\hline
	$K_{m,n}$  & Prime (small $n$-cases)   & Prime by Theorem \ref{main_bipartite_result}\\ \hline
	$K_{3,n}$  & $n = $ 4, 5, 6 																										& $n \geq $ 8 \\ \hline
  $K_{4,n}$  & $n = $ 9       																										& $n \geq $ 13 \\ \hline
  $K_{5,n}$  & $n = $ 14, 15, 16, 18, 19, 20 																			& $n \geq $ 24 \\ \hline
  $K_{6,n}$  & $n = $ 25, 26, 27, 31	 																						& $n \geq $ 35 \\ \hline
  $K_{7,n}$  & $n = $ 36, 37, 38 																									& $n \geq $ 40 \\ \hline
  $K_{8,n}$  & $n = $ 45, 46, 47, 48, 49 																					& $n \geq $ 51 \\ \hline
  $K_{9,n}$  & $n = $ 52 																													& $n \geq $ 58 \\ \hline
  $K_{10,n}$ &  																																	& $n \geq $ 61 \\ \hline
  $K_{11,n}$ & $n = $ 62, 68, 69, 70, 72, 73, 74,							& $n \geq $ 86 \\
  & 78, 79, 80, 81, 82 & \\
  \hline
  $K_{12,n}$ &																																		& $n \geq $ 89 \\ \hline
  $K_{13,n}$ & $n = $ 90, 91, 92																									& $n \geq $ 94 \\
	\hline
	\end{tabular}
\end{center}

\bigskip

We conclude this subsection by applying Theorem~\ref{main_bipartite_result} to find the exact prime labelings for all graphs $K_{m,n}$ when $m = $ 3, 4, 5, or 6. Moreover, we give prime labelings for $K_{m,n}$ for the values of $n$ smaller than $R_{m-1}-m$ for which a prime labeling is possible.

\begin{proposition}
$K_{3,n}$ is prime if $n =4, 5, 6$ or $n \geq 8$.
\end{proposition}

\begin{proof}
Since $R_2=11$, there are at least two primes $p_1, p_2$ in the interval $(\frac{n+3}{2}, n+3]$ for $n \geq 8$. Hence, the sets $A_{3,n} = \{1, p_1, p_2\} \mbox{ and } B_{3,n} = \{1, \ldots, n+3\} \backslash A_{3,n}$ give a prime labeling of $K_{3,n}$ for $n \geq 8$.

If $n = $ 4, 5, or 6, then $A_{3,n} = \{1, 5, 7\}$ and the sets $B_{3,n}= \{2, 3, 4, 6\}, \{2, 3, 4, 6, 8\}$, and $\{2, 3, 4, 6, 8, 9\}$, respectively, give the following prime labelings of $K_{3,4}$, $K_{3,5}$, and $K_{3,6}$:
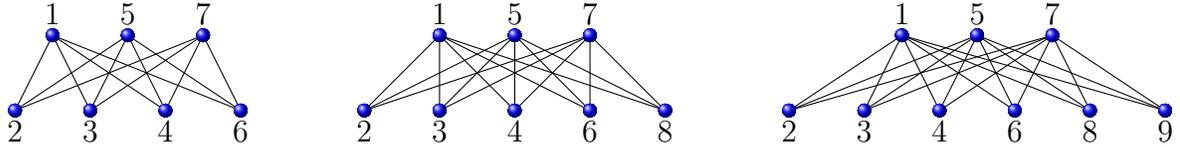
\begin{center}
\begin{tikzpicture}[xscale=1]
\foreach \x in {0.5,1.5,2.5}
\foreach \w in {0,1,2,3}
		\draw (\x,1) -- (\w,0);
\foreach \x in {0,1,2,3}
		\shade[ball color=blue] (\x,0) circle (.5ex);
\foreach \x in {0.5,1.5,2.5}
		\shade[ball color=blue] (\x,1) circle (.5ex);
\node [above] at (0.5,1) {1};
\node [above] at (1.5,1) {5};
\node [above] at (2.5,1) {7};
\node [below] at (0,0) {2};
\node [below] at (1,0) {3};
\node [below] at (2,0) {4};
\node [below] at (3,0) {6};
\end{tikzpicture}
\hspace{.35in}
\begin{tikzpicture}[xscale=1]
\foreach \x in {1,2,3}
\foreach \w in {0,1,2,3,4}
		\draw (\x,1) -- (\w,0);
\foreach \x in {0,1,2,3,4}
		\shade[ball color=blue] (\x,0) circle (.5ex);
\foreach \x in {1,2,3}
		\shade[ball color=blue] (\x,1) circle (.5ex);
\node [above] at (1,1) {1};
\node [above] at (2,1) {5};
\node [above] at (3,1) {7};
\node [below] at (0,0) {2};
\node [below] at (1,0) {3};
\node [below] at (2,0) {4};
\node [below] at (3,0) {6};
\node [below] at (4,0) {8};
\end{tikzpicture}
\hspace{.35in}
\begin{tikzpicture}[xscale=1]
\foreach \x in {1.5,2.5,3.5}
\foreach \w in {0,1,2,3,4,5}
		\draw (\x,1) -- (\w,0);
\foreach \x in {0,1,2,3,4,5}
		\shade[ball color=blue] (\x,0) circle (.5ex);
\foreach \x in {1.5,2.5,3.5}
		\shade[ball color=blue] (\x,1) circle (.5ex);
\node [above] at (1.5,1) {1};
\node [above] at (2.5,1) {5};
\node [above] at (3.5,1) {7};
\node [below] at (0,0) {2};
\node [below] at (1,0) {3};
\node [below] at (2,0) {4};
\node [below] at (3,0) {6};
\node [below] at (4,0) {8};
\node [below] at (5,0) {9};
\end{tikzpicture}
\end{center}
\end{proof}

\begin{proposition}
$K_{4,n}$ is prime if $n=9$ or $n \geq 13$.
\end{proposition}

\begin{proof}
Since $R_3 =17$, there are at least three primes $p_1, p_2, p_3$ in the interval $(\frac{n+4}{2}, n+4]$ for $n \geq 13$. Hence, the sets $A_{4,n} = \{1, p_1, p_2, p_3\} \mbox{ and } B_{4,n} = \{1, \ldots, n+4\} \backslash A_{4,n}$ give a prime labeling of $K_{4,n}$ for $n \geq 13$.

If $n=9$, then the sets $A_{4,9} = \{1, 7, 11, 13\}$ and $B_{4,9} = \{2, 3, 4, 5, 6, 8, 9, 10, 12\}$ give a prime labeling of $K_{4,9}$.
\end{proof}

\begin{proposition}
$K_{5,n}$ is prime if $n=14, 15, 16, 18, 19, 20$ or $n \geq 24$.
\end{proposition}

\begin{proof}
Since $R_4 =29$, there are at least four primes $p_1, p_2, p_3, p_4$ in the interval $(\frac{n+5}{2}, n+5]$ for $n \geq 24$.  Hence, the sets $A_{5,n} = \{1, p_1, p_2, p_3, p_4\} \mbox{ and } B_{5,n} = \{1, \ldots, n+5\} \backslash A_{5,n}$ give a prime labeling of $K_{5,n}$ for $n \geq 24$.

If $n = $ 14, 15, or 16, then choose $A_{5,n} = \{1, 11, 13, 17, 19\}$.  If $n = $ 18, 19, or 20, then choose $A_{5,n} = \{1, 13, 17, 19, 23\}$.  In each case, $B_{5,n}= \{1, \ldots, n+5\} \backslash A_{5,n}$ gives a prime labeling of $K_{5,n}$.
\end{proof}

\begin{proposition}
$K_{6, n}$ is prime if $n=25, 26, 27, 31$ or $n \geq 35$.
\end{proposition}

\begin{proof}
Since $R_5 =41$, there are at least five primes $p_1, p_2, p_3, p_4, p_5$ in the interval $(\frac{n+6}{2}, n+6]$ for $n \geq 35$.  Hence, the sets $A_{6,n} = \{1, p_1, p_2, p_3, p_4, p_5\} \mbox{ and } B_{6,n} = \{1, \ldots, n+6\} \backslash A_{6,n}$ give a prime labeling of $K_{6,n}$ for $n \geq 35$.

If $n = $ 25, 26, or 27, then choose $A_{6,n} = \{1, 17, 19, 23, 29, 31\}$. If $n=31$, then choose $A_{6,31} = \{1, 19, 23, 29, 31, 37\}$.   In each case, $B_{6,n}=\{1, \ldots, n+6 \} \backslash A_{6,n}$ gives a prime labeling of $K_{6,n}$.
\end{proof}

\begin{question}
Is there any predictability as to the values of $n$ smaller than $R_{m-1}-m$ for which $K_{m,n}$ has a prime labeling?  It is interesting to note that when $m=10$ or $m=12$, there are no such values of $n$.
\end{question}

\section{Future work}\label{sec:future_work}

Throughout the paper we have mentioned some questions and conjectures for further research. We conclude here with a few additional open questions.

\begin{question} Is it possible to write every even integer $2n$ in the form $2n=q-p$ where $q$ is prime and $p$ is either 1 or a prime less than $2n+1$? \end{question}

An affirmative answer to this interesting number theory question implies that all ladders are prime and, in fact, have a consecutive cyclic prime labeling. 

\begin{question}
Does there exist an inductive method to get a minimal coprime labeling of the graph $K_{n+1,n+1}$ from a minimal coprime labeling of the graph $K_{n,n}$?
\end{question}

We have observed that minimal coprime labelings of $K_{n,n}$ are sometimes properly contained as subgraphs of a minimal coprime labeling of $K_{n+1,n+1}$.  For example, we have the following labeling of $K_{6,6}$:
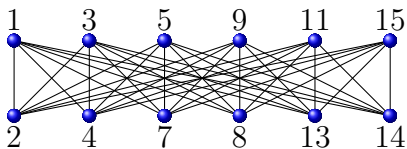
\begin{center}
\begin{tikzpicture}[xscale=1]
\foreach \x in {0,1,2,3,4,5}
\foreach \w in {0,1,2,3,4,5}
		\draw (\x,1) -- (\w,0);
\foreach \x in {0,1,2,3,4,5}
		\shade[ball color=blue] (\x,0) circle (.5ex);
\foreach \x in {0,1,2,3,4,5}
		\shade[ball color=blue] (\x,1) circle (.5ex);
\node [above] at (0,1) {1};
\node [above] at (1,1) {3};
\node [above] at (2,1) {5};
\node [above] at (3,1) {9};
\node [above] at (4,1) {11};
\node [above] at (5,1) {15};
\node [below] at (0,0) {2};
\node [below] at (1,0) {4};
\node [below] at (2,0) {7};
\node [below] at (3,0) {8};
\node [below] at (4,0) {13};
\node [below] at (5,0) {14};
\end{tikzpicture}
\end{center}
Observe that the above graph is properly contained in the following minimal coprime labeling of $K_{7,7}$:
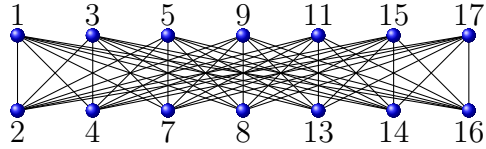
\begin{center}
\begin{tikzpicture}[xscale=1]
\foreach \x in {0,1,2,3,4,5,6}
\foreach \w in {0,1,2,3,4,5,6}
		\draw (\x,1) -- (\w,0);
\foreach \x in {0,1,2,3,4,5,6}
		\shade[ball color=blue] (\x,0) circle (.5ex);
\foreach \x in {0,1,2,3,4,5,6}
		\shade[ball color=blue] (\x,1) circle (.5ex);
\node [above] at (0,1) {1};
\node [above] at (1,1) {3};
\node [above] at (2,1) {5};
\node [above] at (3,1) {9};
\node [above] at (4,1) {11};
\node [above] at (5,1) {15};
\node [above] at (6,1) {17};
\node [below] at (0,0) {2};
\node [below] at (1,0) {4};
\node [below] at (2,0) {7};
\node [below] at (3,0) {8};
\node [below] at (4,0) {13};
\node [below] at (5,0) {14};
\node [below] at (6,0) {16};
\end{tikzpicture}
\end{center}
Also we see that the sets $\popop(\{2,7,13\},6)$ and $\popop(\{2,7,13\},7)$ give the bottom row labelings for $K_{6,6}$ and $K_{7,7}$, respectively.  We observed that in all the examples which we looked at for which we go from a minimal coprime labeling of $K_{n,n}$ to $K_{n+1,n+1}$ in which the above phenomena above arose, the prime numbers in the $\popop$ sets  coincided.

Secondly, another reason to believe that there may exist an inductive method to get from $K_{n,n}$ to $K_{n+1,n+1}$ is the observation that we can always switch a prime number in the top row with one from the bottom row and still have a coprime labeling if and only if there are no multiples of the particular two prime numbers within the set of the other labels in the graph.  For example, we can get from $K_{5,5}$ to $K_{6,6}$ by simply switching the primes 5 and 7 in $K_{5,5}$.  The cost of such a switch is that the label 10 in $K_{5,5}$ cannot also be a label in $K_{6,6}$ since both the values 2 and 5 exist in $K_{6,6}$ as a consequence of the switch.  So we replace 10 with the smallest available prime, namely 13.  Below we illustrate this transition from $K_{5,5}$ (on the left) to $K_{6,6}$ (on the right):
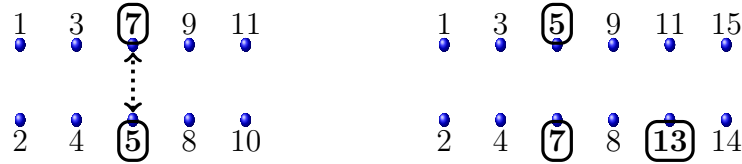
\begin{center}
\begin{tikzpicture}[xscale=.75]
\draw[<->,line width=.5mm, dotted] (2,.1) -- (2,.9);
\foreach \x in {0,1,2,3,4}
		\shade[ball color=blue] (\x,0) circle (.5ex);
\foreach \x in {0,1,2,3,4}
		\shade[ball color=blue] (\x,1) circle (.5ex);
\node [above] at (0,1) {1};
\node [above] at (1,1) {3};
\node [state, above] at (2,1) {\textbf{7}};
\node [above] at (3,1) {9};
\node [above] at (4,1) {11};
\node [below] at (0,0) {2};
\node [below] at (1,0) {4};
\node [state, below] at (2,0) {\textbf{5}};
\node [below] at (3,0) {8};
\node [below] at (4,0) {10};
\end{tikzpicture}
\hspace{.7in}
\begin{tikzpicture}[xscale=.75]
\foreach \x in {0,1,2,3,4,5}
		\shade[ball color=blue] (\x,0) circle (.5ex);
\foreach \x in {0,1,2,3,4,5}
		\shade[ball color=blue] (\x,1) circle (.5ex);
\node [above] at (0,1) {1};
\node [above] at (1,1) {3};
\node [state, above] at (2,1) {\textbf{5}};
\node [above] at (3,1) {9};
\node [above] at (4,1) {11};
\node [above] at (5,1) {15};
\node [below] at (0,0) {2};
\node [below] at (1,0) {4};
\node [state, below] at (2,0) {\textbf{7}};
\node [below] at (3,0) {8};
\node [state, below] at (4,0) {\textbf{13}};
\node [below] at (5,0) {14};
\end{tikzpicture}
\end{center}

\section{Acknowledgements}
The authors thank the AIM (American Institute of Mathematics) and NSF (National Science Foundation) for funding the REUF (Research Experiences for Undergraduate Faculty) program.  Furthermore, we appreciate the hospitality of ICERM (Institute for Computational and Experimental Research in Mathematics), who hosted the REUF program in summer 2012 where the collaborators met and the work on this paper was initiated.

\bigskip
\hrule
\bigskip

\noindent 2010 {\it Mathematics Subject Classification}:
Primary 05C78; Secondary 11A05.

\noindent \emph{Keywords: } Coprime labeling, prime labeling, prime graph, consecutive cyclic prime labeling, bipartite graph, ladder graph, Ramanujan prime.

\bigskip
\hrule
\bigskip

\noindent (Concerned with sequence
\seqnum{A213273},
\seqnum{A213806}, and
\seqnum{A104272}.)

\bigskip
\hrule
\bigskip

\vspace*{+.1in}
\noindent
Received March 11 2016;
revised versions received  June 7 2016.
Published in {\it Journal of Integer Sequences}, June 13 2016.

\bigskip
\hrule
\bigskip

\noindent
Return to
\htmladdnormallink{Journal of Integer Sequences home page}{http://www.cs.uwaterloo.ca/journals/JIS/}.
\vskip .1in

\end{document}